\documentclass[10pt]{amsart}
\usepackage{amssymb, amsthm} 
\usepackage{mathabx} 
\usepackage{enumitem,graphicx}
\usepackage{subfig}
\usepackage[font=small]{caption} 

\graphicspath{{figures/}} 
\def\myScaleVar{1}

\theoremstyle{plain}
    \newtheorem*{theo}{Theorem}
    \newtheorem*{corol}{Corollary}
    \newtheorem{prop}{Proposition}[section]
    \newtheorem{lemma}[prop]{Lemma}

\theoremstyle{definition}
    \newtheorem{defn}[prop]{Definition}
		
\theoremstyle{remark}
    \newtheorem{rem}[prop]{Remark}
    \newtheorem*{question}{Question}

\numberwithin{equation}{section}         

\newcommand{\comment}[1]{}

\newcommand{\R}{\mathbb{R}}

\newcommand{\Z}{\mathbb{Z}}
\newcommand{\N}{\mathbb{N}}
\newcommand{\T}{\mathbb{T}}

\renewcommand{\epsilon}{\varepsilon}
\renewcommand{\phi}{\varphi}
\renewcommand{\setminus}{\smallsetminus}

\DeclareMathOperator{\interior}{int}

\newcommand{\esslune}[1]{#1^{\operatorname{ess}}}
\newcommand{\concat}[2]{{#1}*{#2}}
\newcommand{\floor}[1]{\lfloor {#1} \rfloor}

\title{Peano Curves with Smooth Footprints}

\author[J.~Bochi]{Jairo Bochi}
\address{Facultad de Matem\'aticas, PUC--Chile. Santiago, Chile.}
\email{jairo.bochi@gmail.com}

\author[P.H.~Milet]{Pedro H.~Milet}
\address{Departamento de Matem\'atica, PUC--Rio. Rio de Janeiro, Brazil.}
\email{milet@mat.puc-rio.br}

\date{\today}

\thanks{The first named author is partially supported by project Fondecyt 1140202
and project Anillo ACT1103 (Chile).
The second named author is supported by FAPERJ (Brazil)}


\begin{document}

\begin{abstract}
We construct Peano curves $\gamma : [0,\infty) \to \mathbb{R}^2$
whose ``footprints'' $\gamma([0,t])$, $t>0$, have $C^\infty$ boundaries
and are tangent to a common continuous line field on the punctured
plane $\mathbb{R}^2 \setminus \{\gamma(0)\}$.
Moreover, these boundaries can be taken $C^\infty$-close 
to any prescribed smooth family of nested smooth Jordan curves contracting to a point.
\end{abstract}

\maketitle

\section{Introduction}

A continuous map $\gamma : I \to \R^2$ defined on a nondegenerate interval $I \subseteq \R$
is called a \emph{Peano curve} if its image has nonempty interior.
A lot of water has run under the bridge since Peano established the existence of such curves in 1890. 
Many interesting problems concerning such curves are discussed in the book \cite{sagan1994spacefilling} by H.~Sagan.

By Sard's theorem, Peano curves are non-differentiable.
Nevertheless, they can have smooth ``footprints'',
as a consequence of our main result:

\begin{theo} 
There exist a Peano curve $\gamma : [0,\infty) \to \R^2$
and a continuous line field $\Lambda$ on 
the punctured plane $\R^2 \setminus \{\gamma(0)\}$
such that 
for every $t>0$, the boundary of the set $\gamma([0,t])$ is a $C^\infty$ curve $C_t$ 
containing the point $\gamma(t)$
and tangent to the line field $\Lambda$ at each point.

Moreover, it is possible to choose the Peano curve $\gamma$ so that
each curve $C_t$ is $C^\infty$-close to the circle $x^2+y^2=t^2$.
\end{theo}

Let us state the ``moreover'' part formally:
Given any upper semicontinuous function $k : (0, \infty) \to \N$
and any lower semicontinuous function $\epsilon : (0,\infty) \to (0,\infty)$,
we can choose the Peano curve $\gamma$ in the theorem 
with the following additional property:
for each $t>0$ the curve $C_t = \partial \gamma([0,t])$ 
is the image of a $C^\infty$ embedding $\beta_t$ of the circle $\T := \R / 2\pi\Z$
into $\R^2$ such that
\begin{equation} \label{eq:proximity}
\| \beta_t - \alpha_t \|_{k(t)} < \epsilon(t),
\end{equation}
where $\alpha_t : \T \to \R^2$ is the
embedding $\theta \mapsto (t \cos \theta, t \sin \theta)$
and $\| \mathord{\cdot} \|_k$ is the usual $C^k$ norm;
see \S~\ref{subsec:initialDefinitions} for details.

\medskip

Taking $k \equiv 2$ and a sufficiently small function $\epsilon$,
we can ensure that each curve $C_t$
has everywhere nonzero curvature, and so we obtain:

\begin{corol}
There exist Peano curves $\gamma : [0,1] \to \R^2$ such that each set $\gamma([0,t])$
is convex.
\end{corol}

This result was first obtained by Pach and Rogers in \cite{pachrogers},
and independently by Vince and Wilson in \cite{vincewilson}.
It is inspired by the following question,
attributed by Pach and Rogers to M.~Mihalik and A.~Wieczorek
(see also \cite[Problem~A.37]{CFG91}):

\begin{question}
Is there a Peano curve $\gamma :I \to \R^2$ such that 
the image $\gamma(J)$ of each subinterval $J \subseteq I$ is a convex set?
\end{question}

To our knowledge, this question remains open.
See \cite{PR_monthly} and \cite[Chapter~6]{U06} for related information.

\medskip

Coming back to our theorem, let us observe that 
the family of concentric circles  can be replaced 
by an arbitrary smooth family of nested smooth Jordan curves contracting to a point.
Indeed, it suffices to change coordinates by a suitable diffeomorphism of $\R^2$.

Despite the fact that the boundaries of the ``footprints'' $\gamma([0,t])$ are smooth,
the line field $\Lambda$ is not. Indeed, $\Lambda$ is not even locally Lipschitz,
because otherwise it would be uniquely integrable.
It is also known that generic (in the sense of Baire) continuous
line fields are uniquely integrable (see \cite[pp.~121--123]{Choquet}), which shows that $\Lambda$ is quite pathological.
Other highly non-uniquely integrable line fields are constructed in \cite{BF04};
these are tangent to uncountably many $C^k$-foliations (where $1 \le k <\infty$)
and have the additional property of being H\"older continuous.
It would be interesting to find what the optimal moduli of regularity 
of the line field $\Lambda$ and of the Peano curve $\gamma$ in our Theorem are -- 
in particular, it is not clear whether they can be taken locally H\"older continuous on $\R^2 \setminus \{\gamma(0)\}$ and $(0,\infty)$, respectively.
Let us remark that the optimal H\"older coefficient of a general Peano curve
is $1/2$ (see e.g.\  \cite[Prop.~2.3]{Falconer}).

It seems that it should be possible to extend the theorem to an arbitrary dimension $n \ge 2$,
so that $\partial \gamma([0,t])$ is a $C^\infty$ hypersurface $C^\infty$-close to a sphere, 
and it is tangent to a continuous field of hyperplanes.
Such construction should follow the same ideas of the $n=2$ case,
but since it would be considerably more technical, we will not dwell on it.

\medskip

This paper is divided into two parts:
the longer part, Section~\ref{sec:basics}, is devoted to the proof of a local, more flexible version of the theorem, namely Proposition~\ref{prop:main}.
In the shorter part, Section~\ref{sec:maincurve}, we ``glue'' these local constructions in order to prove the theorem. 

\medskip

This paper is based on the master's dissertation \cite{dissertacao} of the second named author,
which is, in turn, inspired by ideas from \cite{pachrogers,vincewilson}.
We thank the dissertation committee, especially Prof.~Ricardo S\'{a}~Earp 
who posed questions that led to the improvement of the results of the dissertation, presented here.

\section{Local construction}
\label{sec:basics}

\subsection{Initial definitions and statement of the main proposition}
\label{subsec:initialDefinitions}

The aim of this section is to prove Proposition~\ref{prop:main} below,
which constructs special Peano curves whose footprints are \emph{lunes},
objects that are defined as follows:

\begin{defn}\label{def:lune}
Let $f,g:[a,b]\rightarrow \R$ be $C^\infty$ functions such that:
\begin{enumerate}[label=\upshape(\roman*)]
 \item $f$ and $g$ and each of their derivatives coincide in $a$ and in $b$, that is, $f^{(k)}(a) = g^{(k)}(a)$ and $f^{(k)}(b) = g^{(k)}(b)$ for all integer $k \ge 0$;
 \item $f(x) \leq g(x)$, for all $x \in [a,b]$.
 \end{enumerate}
The plane region 
$$
L = L(f,g) = \{(x,y) \mid a \leq x \leq b, \ f(x) \leq y \leq g(x)\}
$$ 
is called a \emph{lune} with \emph{domain} $[a,b]$.
The \emph{support} of a lune $L(f,g)$ is the (open) set 
$$
S_L = \{x \in [a,b] \mid f(x) < g(x)\}.
$$ 
A lune $L$ is said to be \emph{simple} if its support is a nonempty interval.
The \emph{essential part} of the lune $L$ is defined as the closure of the interior of $L$:
$$
\esslune{L} = \esslune{L}(f,g) = \overline{\interior(L)} \, .
$$ 
\end{defn}

\begin{figure}[ht]%
\centering
\includegraphics[height=70pt,width=0.7\columnwidth]{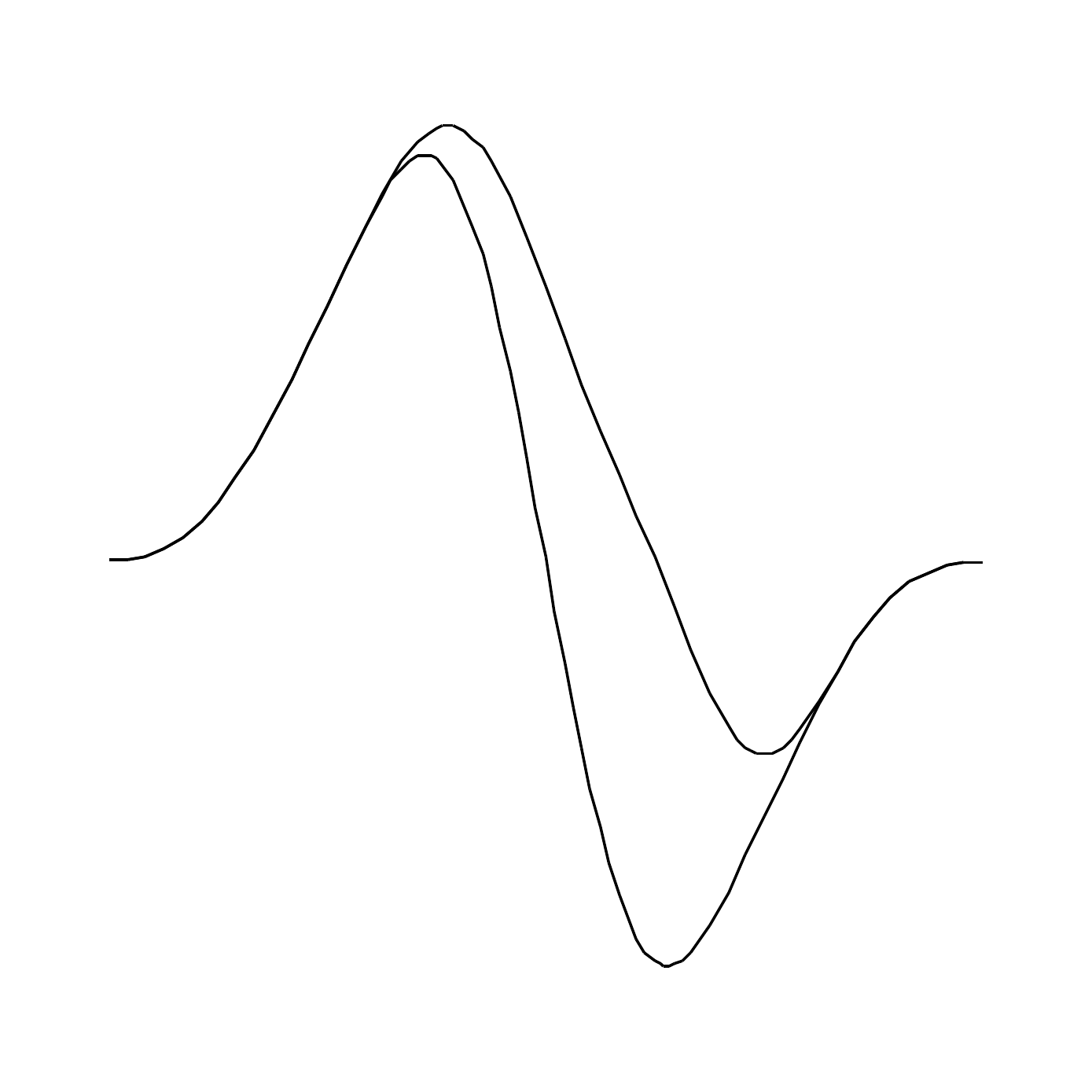}%
\caption{A typical lune in the construction.}%
\label{fig:sample_moon}%
\end{figure}

\begin{lemma}
\label{lemma:simpleLuneEssential}
For a simple lune $L$ whose support is $S_L = (c,d)$, we have 
$$\esslune{L} = \{(x,y) \in L \mid c \leq x \leq d \}.$$
In particular, $\esslune{L}$ is itself a lune, and it is simple.
\end{lemma}
\begin{proof}
If $L = L(f,g)$ has support $S_L = (c,d)$ then its interior is clearly 
$$\{(x,y) \mid  c < x < d, f(x) < y < g(x)\},$$ 
so that $\{(x,y) \in L \mid c \leq x \leq d \} = \overline{\interior(L)} = \esslune{L}$.
\end{proof}

We let $C^\infty([a,b])$ denote the set of all $C^\infty$ functions $F: [a,b] \to \R$.
The \emph{$C^0$ norm} of $F \in C^\infty([a,b])$ is $\|F\|_0 = \sup_{x \in [a,b]} |F(x)|$. For $k \in \N$, the \emph{$C^k$ norm} of $F$ is $$\|F\|_k = \max \left(\|F\|_0, \|F'\|_0, \ldots, \|F^{(k)}\|_0\right).$$ 

A \emph{basic neighborhood} of $F$ is a set of the form
$$
N(F,k,\epsilon) = \big\{ G \in C^\infty([a,b]) \mid \| G-F\|_k < \epsilon \big\}.
$$
We endow the space $C^\infty([a,b])$ with the topology generated by the basic neighborhoods,
called the \emph{$C^\infty$ topology}.

Later on we will work with the space $C^\infty(\T)$ of functions on the circle $\T = \R/2\pi\Z$,
which can be considered as $2\pi$-periodic functions on the line.
The $C^k$ norms and the $C^\infty$ topology on this space are defined analogously.

If $F,G \in C^\infty([a,b])$ (or $C^\infty(\T)$) are such that $F(x) \leq G(x)$ for all $x$, we write $F \leq G$.

We now state our main technical proposition:
\begin{prop}
\label{prop:main}
Let $L = L(f,g)$ be a simple lune defined on an interval $[a,b]$.
Then there exist:
\begin{itemize}
	\item a Peano curve $\gamma:[0,1] \to L$;
	\item a continuous map $t \in [0,1] \mapsto F_t \in C^\infty([a,b])$;
	\item a continuous function $\psi : L \to \R$;
\end{itemize}
with the following properties:
\begin{enumerate}[label=\upshape(\roman*)]
	\item\label{item:floorCeiling} $F_0 = f$, $F_1 = g$;
	\item\label{item:monotonicity} If $t \leq s$ then $F_t \leq F_s$;
	\item\label{item:ceiling} Writing $\gamma(t) = (x(t),y(t))$, we have $y(t) = F_t(x(t))$.
	\item\label{item:psi}  $F'_t(x) = \psi(x,F_t(x))$;
	\item\label{item:footprint} for each $t \in (0,1]$ we have $\gamma([0,t]) = \esslune{L}(f,F_t)$; 
	\item\label{item:endpoints} $\gamma(0) = (a,f(a))$, $\gamma(1) = (b, f(b))$.
\end{enumerate}
\end{prop}

Due to property~\ref{item:ceiling}, the functions $F_t$ are called \emph{ceiling functions}.
By property~\ref{item:psi}, their graphs are tangent to the line field $\Lambda(x,y)$ spanned by the vector field $(1,\psi(x,y))$.
Note that for each $t \in (0,1]$,
the point $\gamma(t)$ belongs to the boundary of $\gamma([0,t])$.
Also, this boundary is everywhere tangent to the line field $\Lambda$,
except for the two extreme points where it is not differentiable.
Finally, note that $\gamma([0,1]) = \esslune{L}$.

Our construction actually yields simple lunes $\esslune{L}(f,F_t)$ for all $t \in (0,1]$, but since this fact is not needed we will not justify it.

\subsection{Lune subdivision processes}
\label{subsec:luneSubdivision}
The proof of Proposition~\ref{prop:main} involves a limiting process on a sequence of subdivisions of the original lune. The basic subdivision processes on a lune are described here.

Throughout the remainder of this section, fix a $C^\infty$ function $\phi:\R \to \R$ with the following properties:
\begin{enumerate}[label=\upshape(\roman*)]
	\item $0 \leq \phi(x) \leq 1$ for all $x \in \R$;
	\item $\phi^{-1}(0) = (-\infty, 1/3]$;
	\item $\phi^{-1}(1) = [2/3,\infty).$
\end{enumerate}

For any $a,b \in \R, a< b$, let $\phi_{a,b}(x) = \phi\left(\frac{x-a}{b-a}\right)$. Notice that 
$\phi_{0,1} = \phi$,  
$\phi_{a,b}^{-1}(0) = \left(-\infty,\frac{2a+b}{3}\right]$ and  $\phi_{a,b}^{-1}(1) = \left[\frac{a+2b}{3}, +\infty\right)$. Also, clearly
$$  \sup_{x \in \R} |\phi_{a,b}^{(k)}(x)| = \sup_{x \in [a,b]} |\phi_{a,b}^{(k)}(x)|,$$
thus $\|\phi_{a,b}\|_k$ is defined for all $k$ by taking the $C^k$ norm of $\phi_{a,b}$ restricted to $[a,b]$. Moreover, since $\phi_{a,b}^{(k)}(x) = \frac{1}{(b-a)^k} \cdot \phi^{(k)}\left(\frac{x-a}{b-a}\right),$ it follows that
$$
\|\phi_{a,b}\|_k = \max_{i \leq k}\frac{\|\phi^{(i)}\|_0}{(b-a)^i}.
$$
Thus, $$\|\phi_{a,b}\|_k \leq \max(1,\|b-a\|^{-k})\|\phi\|_k.$$

\medskip

The $C^k$ norm of a lune $L = L(f,g)$ is defined as $\|L\|_k = \|g - f\|_k$.

We have now set the stage for the definitions of the two basic subdivision processes:

\begin{defn}[Slicing] Let $L = L(f,g)$ be a simple lune, and let $n \in \N$. For $j = 0,1, \ldots, n$ set $h_j = \left(1 - \frac{j}{n}\right)f + \frac{j}{n} g$, and for $i = 1,2, \ldots, n$ set $L_i = L(h_{i-1}, h_i).$ The set of lunes $\{L_1, L_2, \ldots,L_n\}$ is called the $n$-\emph{slicing} of $L$. 
\end{defn}

\begin{prop}
\label{prop:slicing}
The $n$-slicing $\{L_1, \ldots, L_n\}$ of a simple lune $L$
has the following properties:
\begin{enumerate}[label=\upshape(\alph*)]
	\item \label{item:slicingSimpleLune} $S_{L_i} = S_L$ and, in particular, $L_i$ is a simple lune, for each $i \in\{1,2,\ldots, n\}$;
	\item \label{item:slicingIsAPartition} $L = L_1 \cup L_2 \cup \ldots \cup L_{n}$ and $\esslune{L} = \esslune{L_1} \cup \esslune{L_2} \cup \ldots \cup \esslune{L_{n}}$;
	\item \label{item:slicingNorm} $\|L_i\|_k = \frac{1}{n} \|L\|_k$ for each $i \in \{1,\ldots, n\}$, and each $k \in \N$.
\end{enumerate}
\end{prop}
\begin{proof}
Properties~\ref{item:slicingSimpleLune} and \ref{item:slicingIsAPartition} are straightforward from the definition. Property~\ref{item:slicingNorm} follows from the fact that $\|L_i\|_k = \|h_{i+1} - h_i\|_k = \frac{1}{n}\|g - f\|_k$.
\end{proof}

The second basic subdivision process is defined as follows:
 
\begin{defn}[Bipartition] Let $L=L(f,g)$ be a simple lune defined on $[0,1]$ with $S_L = (a,b)$, and set $h(x) = (1 - \phi_{a,b}(x))g(x) + \phi_{a,b}(x)f(x)$ for $0 \leq x \leq 1$. The pair of lunes $\{L(f,h),L(h,g)\}$ is called the \emph{bipartition} of $L$.
\end{defn}  

\begin{prop}
\label{prop:bipartition}
The bipartition $\{L_1, L_2\}$ of a lune $L$ has the following properties:
\begin{enumerate}[label=\upshape(\alph*)]
\item \label{item:bipartitionSupports}
 $S_{L_1} = \left(a, \frac{a+2b}{3}\right)$ and $S_{L_2} = \left(\frac{2a+b}{3}, b\right)$; in particular, $L_1$ and $L_2$ are simple lunes;
\item \label{item:bipartitionIsAPartition} $L = L_1 \cup L_2$ and $\esslune{L} = \esslune{L_1} \cup \esslune{L_2}$;
\item \label{item:bipartitionNorms}
$\max(\|L_1\|_k,\|L_2\|_k) \leq 2^k\|\phi_{a,b}\|_k\|L\|_k$, $\forall k \in \N$.
\end{enumerate}
\end{prop}
\begin{proof}
Properties~\ref{item:bipartitionSupports} and \ref{item:bipartitionIsAPartition} are straightforward from the definition. As for \ref{item:bipartitionNorms}, notice that 
$$h^{(k)}(x) = g^{(k)}(x) + \sum_{i=0}^k {\binom{k}{i}}\phi_{a,b}^{(k-i)}(x)(f - g)^{(i)}(x),$$
so that
$$ \|h^{(k)} - g^{(k)}\|_{0} \leq 2^k \|\phi_{a,b}\|_{k} \|f - g\|_{k} $$
and
$$ 
\|L_2\|_k \leq \max_{i \leq k} (2^i \|\phi_{a,b}\|_{i} \|L\|_{i}) = 2^k \|\phi_{a,b}\|_{k} \|L\|_{k} \, .
$$
The estimate of $\|L_1\|_k$ is analogous.
\end{proof}

\subsection{A family of lunes}
\label{subsec:luneFamily}

Let us begin the proof of Proposition~\ref{prop:main}.
Let a simple lune $L=L(f,g)$ with domain $[a,b]$ be given.
Clearly it is sufficient to consider the case where $S_L=(a,b)$.
By rescaling if necessary, we can assume that $a=0$ and $b=1$.
The functions $f$ and $g$ will be fixed for the remainder of this section.

\medskip

Our recursive construction is based on the following definition:

\begin{defn}
Let $\{m_j\}_{j \geq 1}$ be a sequence of positive integers. The \emph{set of words} with respect to the sequence $\{m_j\}$ is the set
$$ \Omega = \{\omega = (i_1,i_2,\ldots, i_n) \mid n \in \N, \  i_j \in \{1,2, \ldots, m_j\} \text{ for } j=1,2,\ldots,n\},$$
and its elements are called \emph{words with respect to} $\{m_j\}_{j \geq 1}$, or simply \emph{words}. The length of a word $\omega = (i_1, \ldots, i_n)$ is denoted as $|\omega|$ and equals $n$. 

\emph{The set of words with successor} with respect to $\{m_j\}$ is 
$$
\Omega^\ast = \{\omega = (i_1,i_2,\ldots, i_n) \in \Omega \mid i_n < m_n  \}.
$$ 
The \emph{successor} of a word 
$\omega = (i_1, \ldots, i_n) \in \Omega^\ast$ is the word $\omega^+ = (i_1, \ldots, i_{n-1}, i_n+1)$.

Finally, if $\omega_0 = (i_1, \ldots, i_k)$ and $\omega_1 = (j_1, \ldots, j_n)$ are two words, the word 
$$
\concat{\omega_0}{\omega_1} = (i_1, \ldots, i_k, j_1, \ldots j_n)
$$ 
of length $k + n$ is the \emph{concatenation} of $\omega_0$ and $\omega_1$.
\end{defn}

Our first goal is to recursively define both a sequence $\{m_j\}_{j \geq 1}$ of integers and a family $\{L_\omega\}$ of lunes indexed by words with respect to this sequence.

Let us fix a sequence of positive real numbers $\{\epsilon_n\}_{n \geq 2}$ such that 
$$
\sum \epsilon_n < \infty \, .
$$  
Let $m_1 = 1$ and $L_{(1)} = L$. For $k \geq 2$:

\begin{enumerate}[label=Step \arabic*., ref=\arabic*] 

\item Assume we know the values of $m_1, m_2 \ldots, m_{k-1}$ and that $L_\omega$ is defined for every word $\omega$ of length $k-1$ with respect to the sequence $\{m_j\}_{j \geq 1}$ (although this sequence is not yet fully defined, the set of words of length $k-1$ depends only on the $k-1$ first integers in the sequence). 

\item \label{step:choiceOfnk} For every word $\omega$ of length $k-1$, pick $n_\omega$ sufficiently large such that the lunes generated in the $n_\omega$-slicing of $L_\omega$,
$\left\{L_\omega^1, L_\omega^2, \ldots, L_\omega^{n_\omega}\right\}$,
satisfy:
\begin{equation}
\|L_\omega^j\|_k < \frac{\epsilon_{k}}{2^{k} \|\phi_{a,b}\|_{k}}, \quad  j = 1, \ldots, n_\omega,
\label{eq:sufficientlyLargeSlicing}
\end{equation}
where $S_{L_\omega} = (a,b)$ (property~\ref{item:slicingNorm} in Proposition~\ref{prop:slicing} allows us to pick such an $n_\omega$). 

\item Set $n_k = \max\{n_\omega \mid \omega \in \Omega, |\omega| = k-1\}$ and $m_k = 2 n_k$. Notice that, for every word $\omega$ of length $k-1$, the lunes generated in the $n_k$-slicing of $L_\omega$ (which is finer than the $n_\omega$-slicing) satisfy \eqref{eq:sufficientlyLargeSlicing}.

\item For every $\omega$ of length $k-1$, consider the lunes $\left\{L_\omega^1, \ldots, L_\omega^{n_k}\right\}$ generated in the $n_k$-slicing of $L_\omega$. For $j = 1, 2, \ldots, n_k$, let $\{L_\omega^{j,1},L_\omega^{j,2}\}$ be the bipartition of $L_\omega^j$, and set 
$$L_{\concat{\omega}{(2j-1)}} = L_\omega^{j,1} \text{ and } L_{\concat{\omega}{(2j)}} = L_\omega^{j,2} $$

\item We have defined $L_{\omega'}$ for every word $\omega'$ of length $k$. Moreover, by Proposition~\ref{prop:bipartition} and by inequality~\eqref{eq:sufficientlyLargeSlicing},
$\|L_{\omega'}\|_k \leq  \epsilon_k$ whenever $|\omega'| = k$.
\end{enumerate}

\def\myScaleVar{0.41}

\begin{figure}[htp]
 \begin{center}
 \subfloat[A lune is first $3$-sliced (here, schematically, $n_2 = 3$).]{\label{fig:moonSl}\includegraphics[scale=\myScaleVar]{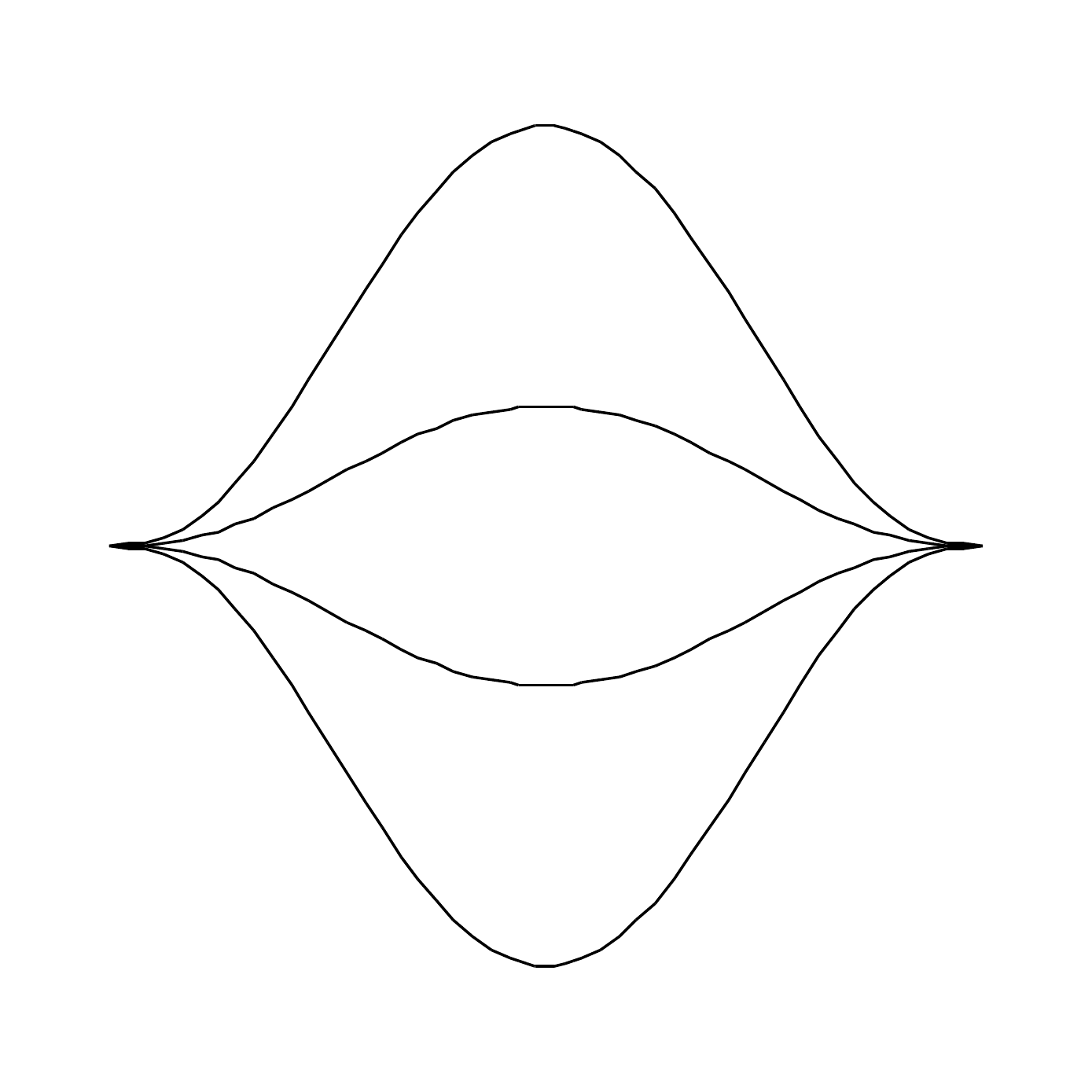}}
 \subfloat[Each slice is bipartitioned, completing the recursion for $k = 2$.]{\label{fig:moonSlBip}\includegraphics[scale=\myScaleVar]{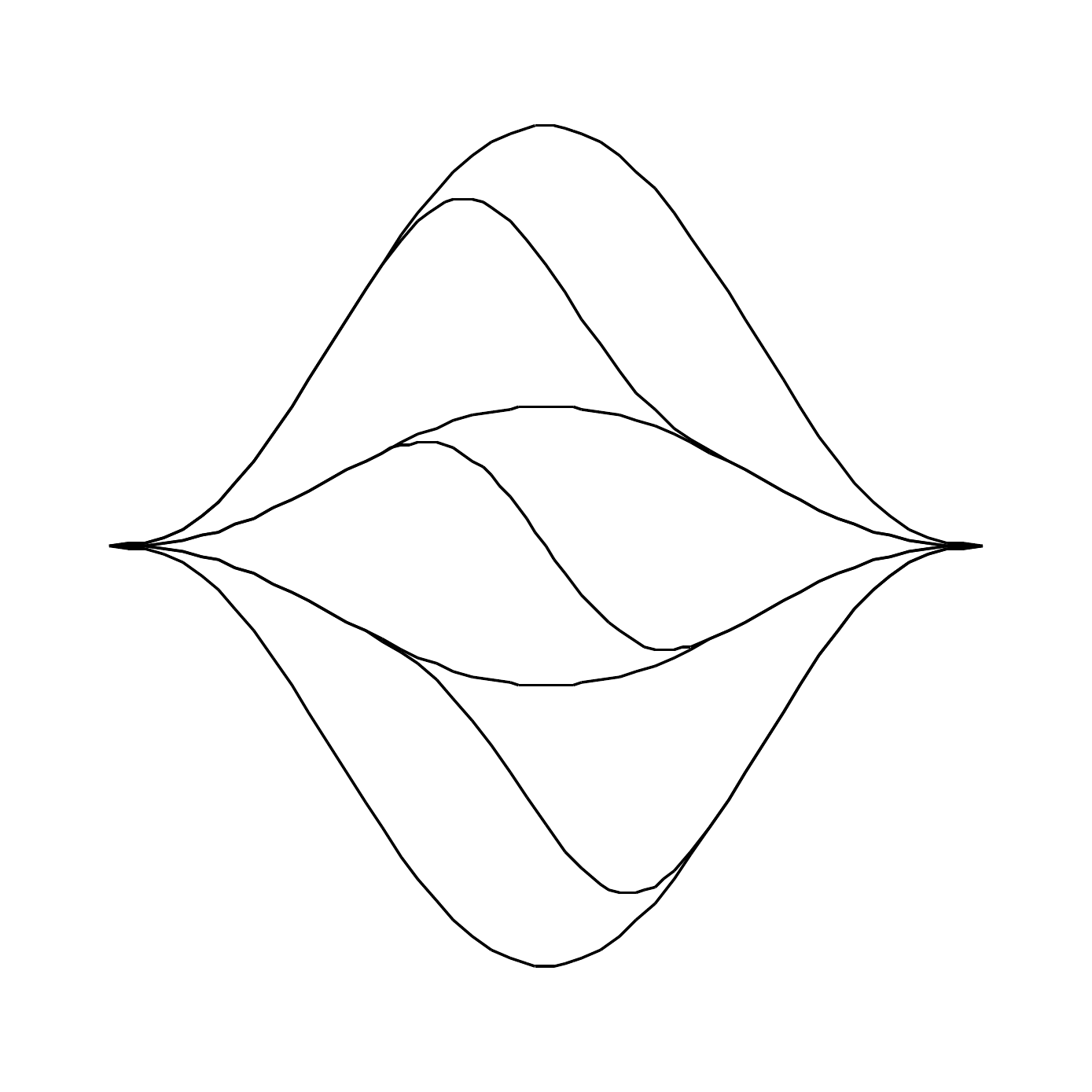}} \\
 \subfloat[Then, assuming $n_3 = 2$, each of the previous lunes is $2$-sliced.]{\label{fig:moonsSlBipSl}\includegraphics[scale=\myScaleVar]{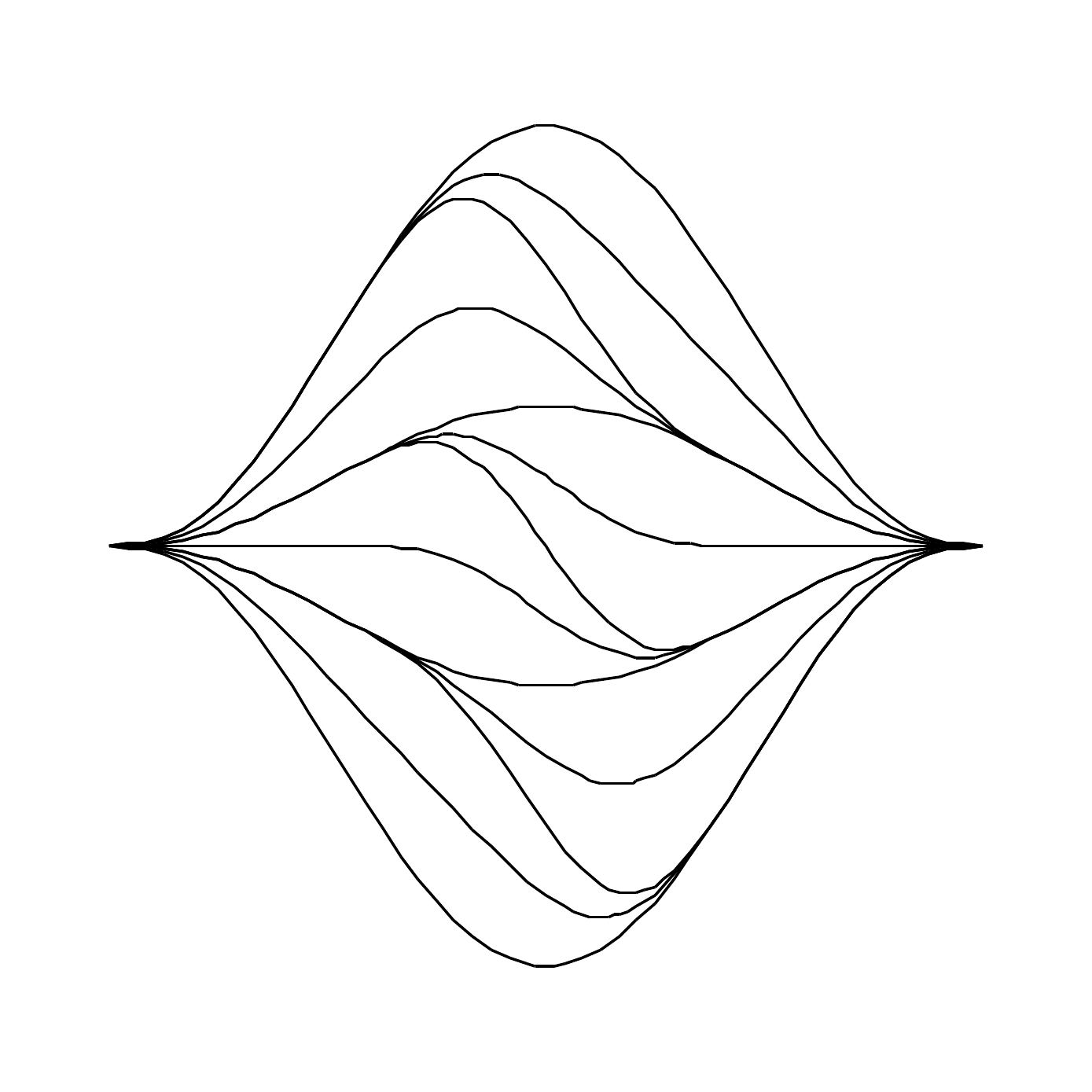}}
 \subfloat[Once again, the slices are bipartitioned, and the recursion with $k=3$ is complete.]{\label{fig:moonSlBipSlBip}\includegraphics[scale=\myScaleVar]{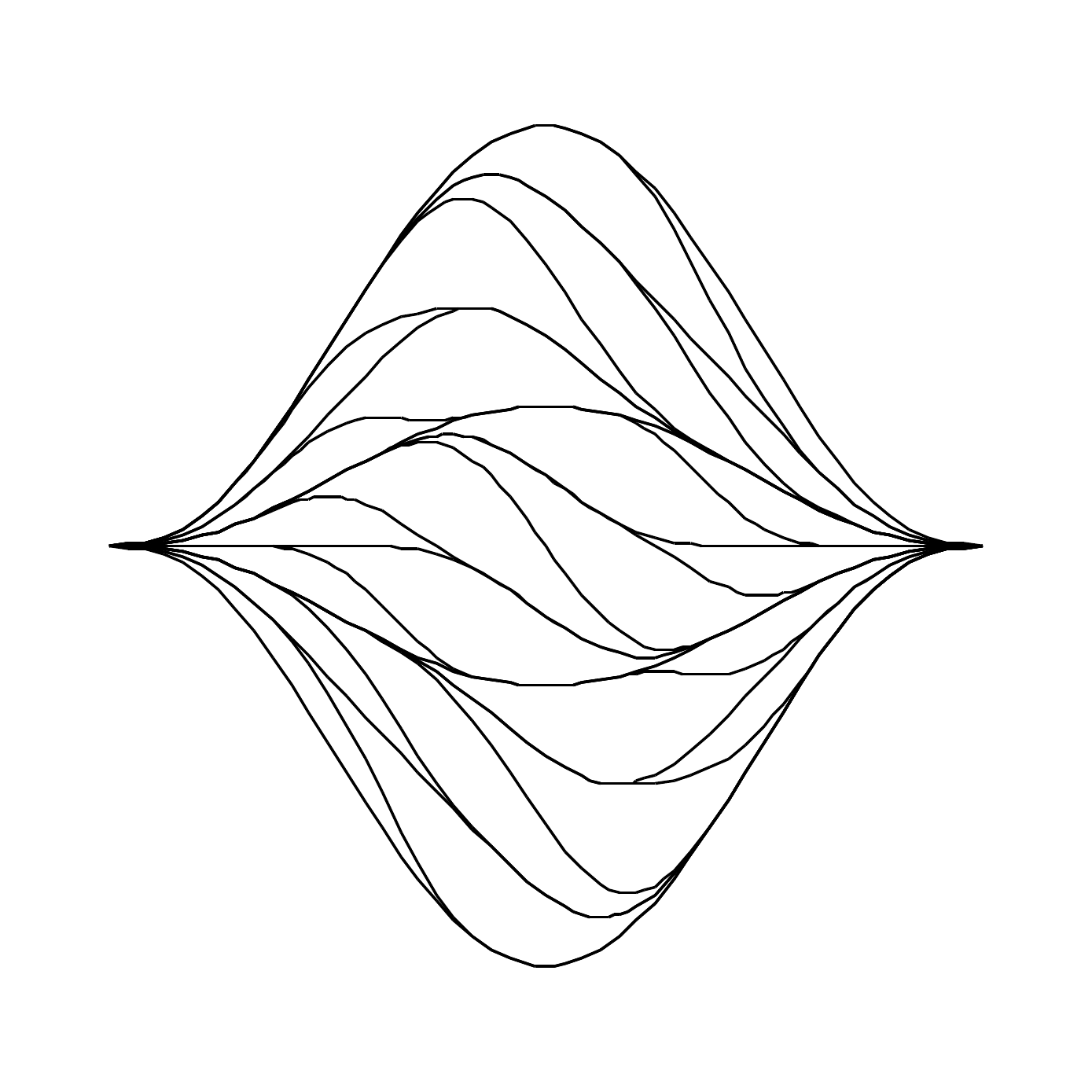}}
  \end{center}
	\caption{Illustration of two recursion steps for the subdivision of a lune. At the end of the recursion step for $k=2$, we have a total of $6$ lunes, $L_{(1,1)}, L_{(1,2)}, \ldots, L_{(1,6)}$. On the other hand, at the end of the recursion for $k = 3$ we have a total of $24$ lunes, ranging from $L_{(1,1,1)}$ up to $L_{(1,6,4)}$; notice that $L_{(1,j,k)}$ is always the result of a subdivision of $L_{(1,j)}$.}
	\label{fig:luneSubdivisionExample}
	\end{figure}
	
	\def\myScaleVar{1}

This subdivision process is illustrated by Figure~\ref{fig:luneSubdivisionExample}. As a result of this construction, we obtain both a sequence $\{m_j\}_{j \geq 1}$ and a family of lunes $\{L_\omega\}_{\omega \in \Omega}$ indexed by words with respect to the aforementioned sequence. In what follows, let $L_\omega = L(f_\omega, g_\omega)$ for all $\omega \in \Omega$. 

\begin{rem}

The following properties hold for the family $\{L_\omega\}_{\omega \in \Omega}$ of lunes:
\begin{enumerate}[label=\upshape(\alph*)] 

\item $\|L_\omega\|_{|\omega|} \leq \epsilon_{|\omega|}$;

\item  $f_{(1)} = f$, $g_{(1)} = g$;

\item $f_{\concat{\omega}{(1)}} = f_\omega$, $g_{\concat{\omega}{(m_{|\omega|+1})}} = g_\omega$;   

\item $g_{\concat{\omega}{ (i)}} = f_{\concat{\omega}{ (i+1)}}$ for $i=1,\ldots, m_{|\omega|+1} - 1$. In other words, $g_\omega = f_{\omega^+}$ for all $\omega \in \Omega^\ast$.

\end{enumerate}
\end{rem}

\begin{lemma}
\label{lemma:controlledNorm}
If a word $\omega \in \Omega$ has length $n$, then for each $1 \leq \ell \leq m_{n}$, 
$$\|f_{\concat{\omega}{(\ell)}} - f_\omega\|_{n} < \epsilon_{n+1} + \epsilon_n $$
\end{lemma}
\begin{proof}
If $\ell$ is odd, then $f_{\concat{\omega}{(\ell)}}$ appeared after a slicing of $L_\omega$, hence 
$$\|f_{\concat{\omega}{(\ell)}} - f_\omega\|_n \leq \|L_\omega\|_n < \epsilon_n.$$
If $\ell$ is even, then
\begin{align*}
\|f_{\concat{\omega}{(\ell)}} - f_\omega\|_n &\leq \|f_{\concat{\omega}{(\ell)}} - f_{\concat{\omega}{(\ell-1)}}\|_n + \|f_{\concat{\omega}{(\ell-1)}} - f_\omega\|_n \\
&\leq \|g_{\concat{\omega}{(\ell-1)}} - f_{\concat{\omega}{(\ell-1)}}\|_n + \|f_{\concat{\omega}{(\ell-1)}} - f_\omega\|_n \\
& \leq \|L_{\concat{\omega}{(\ell-1)}}\|_{n+1} + \|f_{\concat{\omega}{(\ell-1)}} - f_\omega\|_n \\
& < \epsilon_{n+1} + \epsilon_n.  \qedhere
\end{align*} 
\end{proof}

From this point on, we'll assume that from an initial lune $L = L(f,g)$ and a summable positive sequence $\{\epsilon_n\}_{n \geq 2}$, we have obtained, though the procedure described here, a set of words $\Omega$ with respect to a sequence $\{m_k\}_{k \geq 1}$, and a family of lunes $\{L_\omega\}_{\omega \in \Omega}$, with all the properties that were mentioned. 
\subsection{A helpful Cantor set}
Now that we have described the basic subdivision processes, we may begin to describe some auxiliary constructions that play an important part in the definition of a Peano curve (with some special properties) that will cover the initial lune $L$.

First, we will define a Cantor set $K$ through a family of closed intervals $\{J_\omega\}_{\omega \in \Omega}$ indexed by words with respect to $\{m_k\}_{k \geq 1}$. The open intervals that will be removed from $[0,1]$, $\{G_\omega\}_{\omega \in \Omega^\ast}$, indexed by words with successor, will also play an important role. 

First, set $J_{(1)} = [0,1]$. For each $k \geq 2$, assume that $J_\omega$ has been defined for all $\omega \in \Omega$ with $|\omega| = k-1$.
 For each such $\omega$, if $J_\omega = [\alpha, \beta]$ set 
\begin{align*}
J_{\concat{\omega}{(\ell)}} = \left[\alpha + \frac{2\ell-2}{2 m_k - 1}(\beta - \alpha),\alpha + \frac{2\ell-1}{2 m_k - 1}(\beta - \alpha)\right]&, \quad 1 \leq \ell \leq m_k \, ,\\
G_{\concat{\omega}{(\ell)}} = \left(\alpha + \frac{2\ell-1}{2 m_k - 1}(\beta - \alpha),\alpha + \frac{2\ell}{2 m_k - 1}(\beta - \alpha)\right)&, \quad 1 \leq \ell < m_k \, .
\end{align*}
Now set $$
K = \bigcap_{n \in \N} \, \bigcup_{\substack{\omega \in \Omega,\\ |\omega| = n}} J_\omega \, .
$$
Notice that $K$ is a Cantor set, and $[0,1] \setminus K = \bigcup_{\omega \in \Omega^\ast} G_\omega$. The significance of this set $K$ will become apparent later on, but the basic idea is as follows: the curve $\gamma$ that we construct in this section will be such that $\gamma(J_\omega) = \esslune{L_\omega}$ (see Definition~\ref{def:lune}). However, usually $\gamma(\sup J_\omega) \neq \gamma(\inf J_{\omega^+})$, so we connect these ``subcurves'' using $G_\omega$.

\begin{figure}[htb]
\def\svgwidth{\columnwidth}
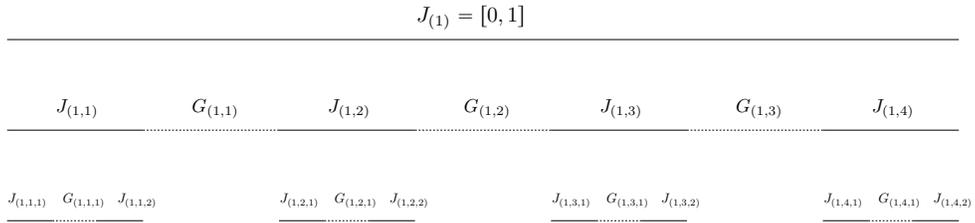%
\caption{Illustration of the first few steps in the construction of $K$. Here $m_2 = 4$ and $m_3 = 2$.}%
\label{fig:cantor}%
\end{figure}

\subsection{The family of ceiling functions}
A point $t \in [0,1]$ belongs to the Cantor set $K$ if and only if there exists a sequence $\{\omega_n\}_{n \geq 1}$ of words with $\omega_1 = (1)$, $\omega_{n+1} = \concat{\omega_n}{(\ell)}$ for some $\ell \in \N$ and such that $t \in J_{\omega_n}$ for each $n$. Moreover, this sequence is unique, and we call it \emph{the defining sequence of} $t$ \emph{in} $K$. For what follows, recall the notation that $L_\omega = L(f_\omega, g_\omega)$.

\begin{lemma}
\label{lemma:uniformlyCauchy}
Let $t \in K$ and let $\{\omega_n\}_{n \geq 1}$ be the defining sequence of $t$ in $K$. Then, for each $k \geq 0$, the sequences of functions $\{f_{\omega_n}^{(k)}\}_n$ and $\{g_{\omega_n}^{(k)}\}_n$ are uniformly Cauchy and
$$ \lim_{n \to \infty} f_{\omega_n}^{(k)} = \lim_{n \to \infty} g_{\omega_n}^{(k)}.$$
\end{lemma}
\begin{proof}
Using Lemma~\ref{lemma:controlledNorm} and noticing that $|\omega_n| = n$, we see that whenever $n \geq k$ we have
$$\|f_{\omega_{n+1}} - f_{\omega_n}\|_k \leq \|f_{\omega_{n+1}} - f_{\omega_n}\|_n < \epsilon_{n+1} + \epsilon_n.$$

If $\epsilon > 0$ is fixed and $N \geq k$ is such that $\sum_{n \geq N} \epsilon_n < \frac{\epsilon}{2}$, then if $m > n > N$,
$$
\|f_{\omega_{m}} - f_{\omega_n}\|_k 
\leq \sum_{j=n}^{m-1} \|f_{\omega_{j+1}} - f_{\omega_j}\|_k 
< 2 \sum_{j=n}^{m} \epsilon_j 
< \epsilon.
$$
Since $\|h^{(k)}\|_0 \leq \|h\|_k$, it follows that $\{f_{\omega_n}^{(k)}\}_n$ is a uniformly Cauchy sequence.
Notice that
$$
\|g_{\omega_n}^{(k)} - f_{\omega_n}^{(k)}\|_0 \leq \|L_{\omega_n}\|_k < \epsilon_n \, .
$$
Hence, the sequence $\{g_{\omega_n}^{(k)}\}_n$ is also uniformly Cauchy, and both sequences have the same limit.
\end{proof}
\begin{defn}
For $t \in K$, let $\{\omega_n\}_{n \geq 1}$ be the defining sequence of $t$ in $K$. The \emph{ceiling function at} $t$ is 
$$F_t = \lim_{n \to \infty} f_{\omega_n} = \lim_{n \to \infty} g_{\omega_n}.$$ 
\end{defn}

A direct consequence of Lemma~\ref{lemma:uniformlyCauchy} is that, for each fixed $t$, $F_t$ is a $C^\infty$ function and $F_t^{(k)} = \lim_{n \to \infty} f_{\omega_n}^{(k)} = \lim_{n \to \infty} g_{\omega_n}^{(k)}.$ 

\begin{lemma}
\label{lemma:ceilingOrder}
If $t,s \in K, t \leq s$, then $F_t \leq F_s$.
\end{lemma}  
\begin{proof}
Let \{$\omega_{t,n}\}_{n \geq 1}, \{\omega_{s,n}\}_{n \geq 1}$ be the defining sequences of $t$ and $s$ in $K$, respectively. Notice that $t < s$ if and only if for some $N > 1$, $\omega_{t,k} = \omega_{s,k}$ whenever $k < N$ but $\omega_{t,N} = \concat{\omega_{t,N-1}}{(j)}, \omega_{s,N} = \concat{\omega_{s,N-1}}{(\ell)}$ with $j < \ell$. Now clearly, for each $n > N$,
$$ f_{\omega_{t,n}} \leq g_{\omega_{t,N}} \leq f_{\omega_{s,N}} \leq f_{\omega_{s,n}},$$ because $L_{\omega_{t,n}}$ is a subdivision of $L_{\omega_{t,N}}$.   
\end{proof}
\begin{prop}
\label{prop:continuityCeilingAtK}
The function $t \in K \mapsto F_t \in C^\infty([0,1])$ is continuous.
\end{prop} 
\begin{proof}
Let $k \in \N$, fix $\epsilon > 0$, and let $N > k$ be such that 
$$ \sum_{n=N}^\infty \epsilon_n < \frac{\epsilon}{4}.$$

Recall that the length of the interval $J_\omega$ is the same for every word $\omega$ such that $|\omega| = N$; call this length $\delta_N$. If $t,s \in K$, with $|t-s| < \delta_N$, then clearly $t,s \in J_\omega$ for some $\omega$ with $|\omega| = N$.
Let $\omega_{t,n}, \omega_{s,n}$ be the defining sequences of $t$ and $s$ in $K$ (notice that $\omega_{t,N} = \omega_{s,N} = \omega$). Since by Lemma~\ref{lemma:controlledNorm}
$$\|f_{\omega_{t,n}} - f_{\omega}\|_k \leq \sum_{j=N}^{n-1} \|f_{\omega_{t,j+1}} - f_{\omega_{t,j}}\|_k < 2 \sum_{j=N}^{n-1} \epsilon_j $$
whenever $n > N$, it follows that by making $n \to \infty$ we have 
$$ \| F_t - f_\omega \|_k \leq 2 \sum_{j=N}^\infty \epsilon_j < \frac{\epsilon}{2}$$
and, analogously, $\| F_s - f_\omega \|_k < \frac{\epsilon}{2}.$ Therefore,
$$\|F_t - F_s\|_k \leq \| F_t - f_\omega \|_k + \| f_\omega - F_s \|_k < \epsilon,$$
which completes the proof.   
\end{proof} 

\begin{lemma}
\label{lemma:ceilingEndsMeet}
Given $\omega \in \Omega^\ast$, suppose $G_\omega = (\alpha,\beta)$. Then $\alpha, \beta \in K$ and $F_\alpha = F_\beta$.
\end{lemma} 
\begin{proof}
Suppose $|\omega| = N$. By the definition of $K$, it is clear that $\alpha \in J_\omega, \beta \in J_{\omega^+}.$ In fact, for each $n > N$ we have
$$\alpha \in J_{\concat{\omega}{(m_{N+1},m_{N+2}, \ldots, m_n)}}, \beta \in J_{\concat{\omega^+}{\underbrace{\scriptstyle(0,0, \ldots, 0)}_{n-N}}}.$$
Since $$g_{\concat{\omega}{(m_{N+1},m_{N+2}, \ldots, m_n)}} = g_{\omega} = f_{\omega^+} = f_{\concat{\omega^+}{(0,0, \ldots, 0)}},$$ 
it follows that
$$F_\alpha = \lim_{n \to \infty} g_{\omega_{\alpha,n}} = \lim_{n \to \infty} f_{\omega_{\beta,n}} = F_\beta,$$
where $\omega_{\alpha,n}$ and $\omega_{\beta,n}$ are the elements of the defining sequences of $\alpha$ and $\beta$ in $K$.
\end{proof}
What Lemma~\ref{lemma:ceilingEndsMeet} implies is that the function $t \mapsto F_t$ can be extended to the interval $[0,1]$ in a natural way: for a point $t \notin K$, there exists a unique $G_\omega = (\alpha, \beta)$ such that $t \in G_\omega$. Then $F_t = F_\alpha = F_\beta$ is the ceiling function at $t$. 

\begin{prop} \label{prop:continuityCeiling}
The function $t \in [0,1] \mapsto F_t \in C^\infty([0,1])$ is continuous.
\end{prop}
\begin{proof}
Given $k \in \N$, fix $\epsilon > 0$ and let $N$ and $\delta_N$ be as in the proof of Proposition~\ref{prop:continuityCeilingAtK}.

Suppose $t < s$ and $|t - s| < \delta_N$. If $t,s \in K$, we already know that $\|F_t - F_s\|_k < \epsilon$. Otherwise, there exist $\alpha, \beta \in K$ such that $t \leq \alpha \leq \beta \leq s$ and $F_t = F_\alpha, F_s = F_\beta$. Since $|\alpha - \beta| \leq |t - s| < \delta_N$, we're done. 
\end{proof}

\subsection{The function $\psi$ and the associated line field}

\begin{lemma}
\label{lemma:ceilingHeight}
Given $(x,y) \in L$, there exists $t \in [0,1]$ (not necessarily unique) such that $y = F_t(x)$. Moreover, if $y = F_{t_1}(x) = F_{t_2}(x)$, then $F_{t_1}'(x) = F_{t_2}'(x)$.  
\end{lemma}
\begin{proof}
Take $(x,y) \in L$, i.e., such that $x \in [0,1], f(x) \leq y \leq g(x)$. Since from the construction in \S~\ref{subsec:luneFamily} we know that for each $n$,
$$\bigcup_{1 \leq i \leq m_n} L_{\concat{\omega}{ (i)}} = L_\omega,$$
it follows that there exists a sequence $\{\omega_n\}_{n \geq 1}$ (not necessarily unique) such that $\omega_1 = (1), \omega_{n} = \concat{\omega_{n-1}}{(i)}$ for some $i \in \{1, 2, \ldots, m_n\}$ with $f_{\omega_n}(x) \leq y \leq g_{\omega_n}(y)$. If $t$ is the only element in $\bigcap_{n \geq 1} J_{\omega_n} \subseteq K$, then 
$$F_t(x) = \lim_{n \to \infty} f_{\omega_n}(x) = \lim_{n \to \infty} g_{\omega_n}(x) = y,$$
proving the first part.

Now suppose $t_1, t_2 \in [0,1]$ are such that $y = F_{t_1}(x) = F_{t_2}(x),$ and assume $t_1 \leq t_2$. By Lemma~\ref{lemma:ceilingOrder}, $F_{t_1} \leq F_{t_2}$, so that $x$ is a local maximum of the function $F_{t_1} - F_{t_2}$. Therefore, $F_{t_1}'(x) - F_{t_2}'(x) = 0$.
\end{proof}

For each $(x,y) \in L$, take $t \in [0,1]$ such that $y = F_t(x)$, and set $\psi(x,y) = F_t'(x)$. By Lemma~\ref{lemma:ceilingHeight}, this is well defined (not depending on the choice of $t$).

\begin{prop}
\label{prop:continuityOfPsi}
The function $\psi: L \to \R$ is continuous. 
\end{prop}
\begin{proof}
Let $(x,y) \in L$ and let $(x_n,y_n) \in L$ be a sequence such that $(x_n, y_n) \to (x,y).$ By Lemma~\ref{lemma:ceilingHeight} and the previous paragraph, $y_n = F_{t_n}(x_n)$ for some $t_n$, and $\psi(x_n,y_n) = F_{t_n}'(x_n)$.

Suppose, by contradiction, that $\psi(x_n,y_n) \not\to \psi(x,y)$. In other words, $F_{t_n}'(x_n)\not\to F_t'(x)$, where $t$ is such that $y = F_t(x)$. By passing to a subsequence if necessary, we may assume that $|F_{t_n}'(x_n) - F_t'(x)| \geq \epsilon$ for some $\epsilon > 0$. 

Let $\{t_{n_k}\}$ be a convergent subsequence of $\{t_n\}$, such that $t_{n_k} \to t^\ast$. By Proposition~\ref{prop:continuityCeiling} and the continuity of $F_{t^\ast}$, the distance
\begin{align*}
|F_{t_{n_k}}(x_{n_k}) - F_{t^\ast}(x)| &\leq |F_{t_{n_k}}(x_{n_k}) - F_{t^\ast}(x_{n_k})| + |F_{t^\ast}(x_{n_k}) - F_{t^\ast}(x)|\\
&\leq \|F_{t_{n_k}} - F_{t^\ast}\|_0 + |F_{t^\ast}(x_{n_k}) - F_{t^\ast}(x)| 
\end{align*}
can be made arbitrarily small as $k \to \infty$, so that $F_{t_{n_k}}(x_{n_k}) \to F_{t^\ast}(x)$ and $F_{t^\ast}(x) = F_t(x)$. By Lemma~\ref{lemma:ceilingHeight}, $F_{t^\ast}'(x) = F_t'(x)$. If we pick $k$ sufficiently large such that  
$\|F_{t_{n_k}} - F_{t^\ast}\|_1 < \frac{\epsilon}{2}$ and  $|F_{t^\ast}'(x_{n_k}) - F_{t^\ast}'(x)| < \frac{\epsilon}{2}$, we obtain
$$|F_{t_{n_k}}'(x_{n_k}) - F_{t}'(x)| < \epsilon,$$
which is a contradiction.
\end{proof} 

We then set $\Lambda(x,y)$ as the line whose direction vector is $(1,\psi(x,y))$, for each $(x,y) \in L$. By Proposition~\ref{prop:continuityOfPsi}, this is a continuous line field.

\subsection{A sequence of curves} 
 We now proceed to the construction of a sequence $\gamma_n$ of curves that converges uniformly to a Peano curve $\gamma$, such that each $\gamma_n$ is tangent to the line field $\Lambda(x,y)$.  

For an interval $I=[\alpha,\beta]$ and for $a,b \in \R$, let $\psi_{I,a,b}:I \to \R$ be a $C^\infty$ strictly monotone function such that
\begin{enumerate}[label=\upshape(\roman*)]
 \item ${\displaystyle \lim_{t \searrow \alpha} \psi_{I,a,b}(t) = a}$, 
 ${\displaystyle \lim_{t \nearrow \beta}\psi_{I,a,b}(t) = b}$. 
 \item ${\displaystyle \lim_{t \searrow \alpha} \psi_{I,a,b}^{(k)}(t) = \lim_{t \nearrow \beta} \psi_{I,a,b}^{(k)}(t) = 0}$ for every $k \geq 1$.
 \end{enumerate}
Additionally, for any $h: [0,1] \to \R$ let $\Gamma_h$ be the graph of $h$, parametrized in the obvious way, i.e., $\Gamma_h(t) = (t,h(t))$ for $t \in [0,1]$.

For each $\omega \in \Omega$, write $S_{L_\omega} = (a_{\omega}, b_{\omega}),$ and let 
$$
\gamma_1(t) = \Gamma_{f_{(1)}}\left(\psi_{J_{(1)},a_{(1)},b_{(1)}}(t)\right)
$$ 
for $t \in [0,1]$. Recursively, set for $n \geq 2$:
$$
\gamma_n(t) = \begin{cases} \gamma_{n-1}(t), &\text{if } t \in \bigcup_{\omega \in \Omega^\ast, |\omega| < n}G_\omega, \\
 \Gamma_{f_\omega}\left(\psi_{J_\omega, a_\omega, b_\omega}(t)\right), &\text{if } t \in J_\omega \text{ for some } \omega \in \Omega, |\omega| = n,\\
\Gamma_{g_\omega}\left(\psi_{G_\omega, b_\omega, a_{\omega^+}}(t)\right), &\text{if } t \in G_\omega \text{ for some } \omega \in \Omega^\ast, |\omega| = n. 
\end{cases}
$$

\begin{figure}[ht]%
\includegraphics[width=0.6\columnwidth]{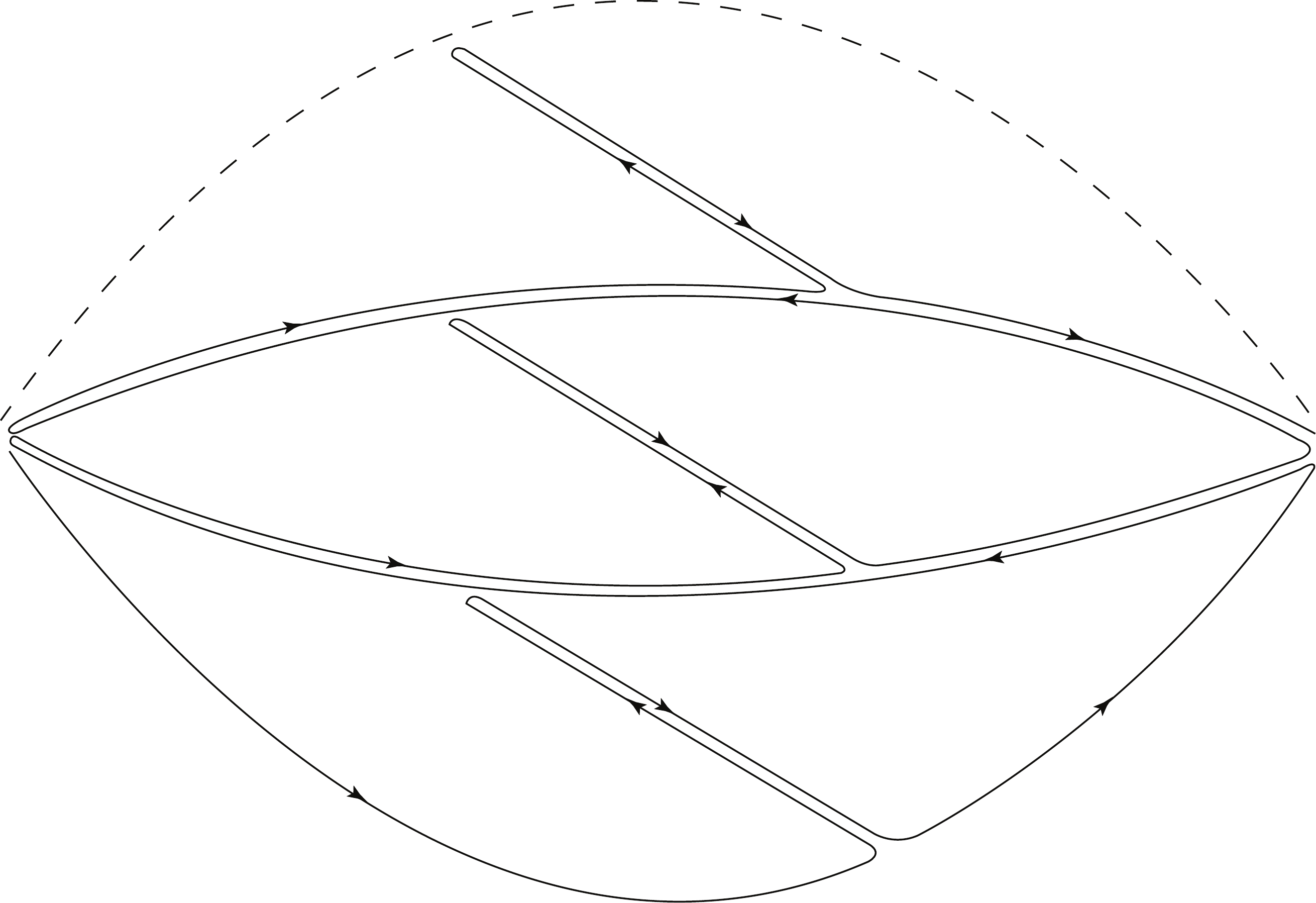}%
\caption{Schematic drawing of the curve $\gamma_2$, assuming $m_2 = 6.$}%
\label{fig:gamma2}%
\end{figure}

\begin{lemma}
\label{lemma:gammaNJ}
If $|\omega| = n$, then 
\begin{enumerate}[label=\upshape(\roman*), parsep=0.5pc]
	\item \label{item:gammaNJ} 
	$\gamma_{n+1}(J_{\concat{\omega}{(i)}}) \subseteq \esslune{L_{{\concat{\omega}{(i)}}}} \text{ for any } 1 \leq i \leq m_{n+1}$;
	\item \label{item:gammaNGodd} 
	$\gamma_{n+1}(G_{\concat{\omega}{(i)}}) \subseteq \esslune{L_{{\concat{\omega}{(i)}}}} \text{ for odd } i, 1 \leq i < m_{n+1}$;
	\item \label{item:gammaNGeven} 
	$\gamma_{n+1}(G_{\concat{\omega}{(i)}}) \subseteq \esslune{L_{{\concat{\omega}{(i-1)}}}} \cup \esslune{L_{{\concat{\omega}{(i)}}}} \text{ for even } i, 1 \leq i < m_{n+1}$. 

As a consequence, $\gamma_{m}(J_\omega) \subseteq \esslune{L_\omega}$ for all $m \geq n$.
\end{enumerate}
\begin{proof}
If $t \in J_{\concat{\omega}{(i)}}$, then $$\gamma_{n+1}(t) = \Gamma_{f_{\concat{\omega}{(i)}}}(\psi_{J_{\concat{\omega}{(i)}}, a_{\concat{\omega}{(i)}}, b_{\concat{\omega}{(i)}}}(t)).$$ Since $a_{\concat{\omega}{(i)}} \leq \psi_{J_{\concat{\omega}{(i)}}, a_{\concat{\omega}{(i)}}, b_{\concat{\omega}{(i)}}}(t) \leq b_{\concat{\omega}{(i)}},$ Lemma~\ref{lemma:simpleLuneEssential} implies property~\ref{item:gammaNJ}.

If $t \in G_{\concat{\omega}{(i)}}$ with odd $i$, then $L_{\concat{\omega}{(i)}}$ is the first lune in a bipartition of one of the slices of $L_\omega$, hence $b_{\concat{\omega}{(i)}} = (a_\omega + 2b_\omega)/3$ and $a_{\concat{\omega}{(i+1)}} = (2a_\omega + b_\omega)/3$ (see Proposition~\ref{prop:bipartition}). Since 
$$\gamma_{n+1}(t) = \Gamma_{g_{\concat{\omega}{(i)}}}(\psi_{G_{\concat{\omega}{(i)}}, b_{\concat{\omega}{(i)}}, a_{\concat{\omega}{(i+1)}}}(t)),$$ 
property~\ref{item:gammaNGodd} follows.

If $t \in G_{\concat{\omega}{(i)}}$ with even $i$ and $i < m_{n+1}$, $L_{\concat{\omega}{(i)}}$ is the second lune in a bipartition, hence $b_{\concat{\omega}{(i)}} = b_\omega$ and $a_{\concat{\omega}{(i+1)}} = a_\omega = a_{\concat{\omega}{(i-1)}}$. 
Therefore, if the value $\psi_{G_{\concat{\omega}{(i)}}, b_{\concat{\omega}{(i)}}, a_{\concat{\omega}{(i+1)}}}(t)$ is greater than or equal to $a_{\concat{\omega}{(i)}}$ then $\gamma_{n+1}(t) \in L_{\concat{\omega}{(i)}}$; on the other hand, if that value is less than or equal to $a_{\concat{\omega}{(i)}} \leq b_{\concat{\omega}{(i-1)}}$ then $\gamma_{n+1}(t) \in L_{\concat{\omega}{(i-1)}}$ (recall that if $x \leq a_{\concat{\omega}{(i)}}$, $g_{\concat{\omega}{(i)}}(x) = f_{\concat{\omega}{(i)}}(x) = g_{\concat{\omega}{(i-1)}}(x)$), thus property~\ref{item:gammaNGeven}  follows.

To see the consequence, notice that the three properties imply that 
for any $\omega \in \Omega$ with $|\omega| = n$, 
$\gamma_{n+1}(J_{\concat{\omega}{(i)}}),\gamma_{n+1}(G_{\concat{\omega}{(i)}}) \subseteq \esslune{L_\omega}$ for each $i$. 
In particular, for any $\omega' \in \Omega$ with $|\omega'| = k$, 
$\gamma_{n+k+1}(J_{\concat{\concat{\omega}{\omega'}}{(i)}}), \gamma_{n+k+1}(G_{\concat{\concat{\omega}{\omega'}}{(i)}}) \subseteq \esslune{L_{\concat{\omega}{\omega'}}} \subseteq \esslune{L_\omega}$. 
Since for each $k$ we have
$$
J_\omega = \bigcup_{\substack{\omega' \in \Omega\\ |\omega'| = k}} J_{\concat{\omega}{\omega'}} \cup \bigcup_{\substack{\omega' \in \Omega^\ast\\ |\omega'| = k}} G_{\concat{\omega}{\omega'}},
$$
it follows that $\gamma_{n+k}(J_\omega) \subseteq \esslune{L_\omega}$.
\end{proof}
\end{lemma}

\subsection{The Peano curve} 

We now wish to show that the curves $\gamma_n$ defined above converge uniformly to some curve $\gamma$.

\begin{lemma}
There exists a sequence $D_n \to 0$ such that for each $\omega \in \Omega$, the diameter $D_\omega$ of $\esslune{L_\omega}$ satisfies $D_\omega^2 \leq D_{|\omega|}$.
\end{lemma}
\begin{proof}
Let $(x_1,y_1), (x_2,y_2) \in \esslune{L_\omega}$, and assume without loss of generality that $y_1 \leq y_2$. By Lemma~\ref{lemma:simpleLuneEssential}, $a_\omega \leq x_1,x_2 \leq b_\omega$ and $f_\omega(x_1) \leq y_1 \leq y_2 \leq g_\omega(x_2)$. Hence,
\begin{align*}
\operatorname{dist}((x_1,y_1),(x_2,y_2))^2 &= (x_2 - x_1)^2 + (y_2- y_1)^2\\
&\leq (b_\omega - a_\omega)^2 + (g_\omega(x_2) - f_\omega(x_1))^2\\
&\leq (b_\omega - a_\omega)^2 + (g_\omega(x_2) - f_\omega(x_2) + f_\omega(x_2) - f_\omega(x_1))^2\\
&\leq (b_\omega - a_\omega)^2 + \left(|g_\omega(x_2) - f_\omega(x_2)| + \left|\int_{x_1}^{x_2}f_\omega'(t) dt\right|\right)^2\\
&\leq (b_\omega - a_\omega)^2 + \left(\|L_\omega\|_0 + \|f_\omega\|_1 (b_\omega - a_\omega)\right)^2.
\end{align*}

Suppose $|\omega| = n$. Clearly $\|L_\omega\|_0 \leq \|L_\omega\|_n < \epsilon_n$, by the construction of the family of lunes. Moreover, also by construction (and by Proposition~\ref{prop:bipartition}), $b_\omega - a_\omega = (2/3)^n(b_{(1)} - a_{(1)}) \leq (2/3)^n$. Finally, let $M = \sup_{t \in [0,1]} \|F_t\|_1$ (which is finite by Proposition~\ref{prop:continuityCeiling}). Since 
$$f_\omega = \lim_{k \to \infty} f_{\omega * \underbrace{\scriptstyle(1,1, \ldots, 1)}_k } = F_t$$
for some $t \in [0,1]$, it follows that $\|f_\omega\|_{1} \leq M$. In other words,
if we take $$D_n = \left(\frac{2}{3}\right)^{2n} + \left(\epsilon_n + M\left(\frac{2}{3}\right)^n\right)^{2} \to 0,$$
we're done.  
\end{proof}

\begin{lemma} \label{lemma:gamma}
The curves $\gamma_n$ form a uniformly Cauchy sequence,
and in particular converge uniformly to a continuous curve $\gamma$. 
\end{lemma}
\begin{proof}
Let $\epsilon > 0$ and take $N$ such that 
$D_N < \epsilon.$  
We wish to show that whenever $m,n \geq N$, we have $|\gamma_m(t) - \gamma_n(t)| < \epsilon$ for each $t \in [0,1]$. 

In fact, if $t \in \bigcup_{\omega \in \Omega^\ast, |\omega| \leq N} G_\omega$ this is clear, because $\gamma_k(t) = \gamma_N(t)$ for all $k \geq N$. On the other hand, if $t \in J_\omega$ for some $\omega, |\omega| = N$, then Lemma~\ref{lemma:gammaNJ} shows that $\gamma_m(t)$, $\gamma_n(t) \in \esslune{L_\omega}$, thus
$|\gamma_m(t) - \gamma_n(t)| \leq D_N < \epsilon.$
\end{proof}

We now wish to relate the footprint $\gamma([0,t])$ with the ceiling function $F_t$.
We need a few results first:
\begin{lemma}
\label{lemma:definingSequenceLune}
Let $\omega \in \Omega$, $(x,y) \in \esslune{L_\omega}$. There exists a sequence $\{\omega_n\}_{n \geq |\omega|}$ such that $\omega_{|\omega|} = \omega$, $\omega_{n+1} = \concat{\omega_n}{(\ell)}$ for some $\ell \in \N$ and $(x,y) \in \esslune{L_{\omega_n}}$ for each $n$. 

Moreover, there exists $t \in J_\omega \cap K$ such that the defining sequence $\{\omega_{t,n}\}_{n \geq 1}$ of $t$ satisfies $\omega_{t,n} = \omega_n$ whenever $n \geq |\omega|$.
\end{lemma}
\begin{proof}
The first part is immediate from the fact that 
$$
\esslune{L_{\omega_n}} = \bigcup_{1 \leq i \leq m_{n+1}} \esslune{L_{\concat{\omega_n}{(i)}}}.
$$ 
For the second part, we take $t$ to be the single element of $\bigcap_{n \geq |\omega|} J_{\omega_n}$.
\end{proof}
The sequence in Lemma~\ref{lemma:definingSequenceLune} is not necessarily unique, and we call every such sequence a \emph{defining sequence} of $(x,y)$ in $L_\omega$.

\begin{lemma}
\label{lemma:gammaJ}
For each $\omega \in \Omega$, $\gamma(J_\omega) = \gamma(J_\omega \cap K) = \esslune{L_\omega}$.
\end{lemma}
\begin{proof}
Lemma~\ref{lemma:gammaNJ} implies that $\gamma(J_\omega) \subseteq \esslune{L_\omega}$. For the other direction, take for each $(x,y) \in \esslune{L_\omega}$, $\{\omega_n\}_{n \geq |\omega|}$ and $t \in J_\omega \cap K$ as in Lemma~\ref{lemma:definingSequenceLune}. Since $\gamma(t) \in \esslune{L_{\omega_n}}$ for each $n$ (and this is a sequence of nested compact sets), it follows that $\gamma(t) = (x,y)$. Therefore, $\gamma(J_\omega) \subseteq \esslune{L_\omega} \subseteq \gamma(J_\omega \cap K) \subseteq \gamma(J_\omega)$ and we're done.
\end{proof}

\begin{lemma}
\label{lemma:onlyGammaKMatters}
For each $t \in [0,1]$, $\gamma([0,t]) = \gamma(K \cap [0,t])$.
\end{lemma}
\begin{proof}
We need to show that for each $t \notin K$, there exists $s < t, s \in K$ with $\gamma(s) = \gamma(t)$. Take $t \notin K$, so that $t \in G_{\concat{\omega}{(i)}}$ for some $\omega \in \Omega^\ast, |\omega| = n$, $i < m_{n+1} - 1$. By Lemma~\ref{lemma:gammaNJ}, $\gamma(t) = \gamma_{n+1}(t) \in \esslune{L_{{\concat{\omega}{(i-1)}}}} \cup \esslune{L_{{\concat{\omega}{(i)}}}}$ (or simply $\esslune{L_{{\concat{\omega}{(i)}}}}$ if $i = 1$).
By Lemma~\ref{lemma:gammaJ}, 
$ \gamma(J_{\concat{\omega}{(i-1)}} \cap K) \cup \gamma(J_{\concat{\omega}{(i)}} \cap K) = \esslune{L_{{\concat{\omega}{(i-1)}}}} \cup \esslune{L_{{\concat{\omega}{(i)}}}},$
and this yields the result.  
\end{proof}

\begin{lemma}
\label{lemma:gammaTinTheCeiling}
For $t \in [0,1]$, if $\gamma(t) = (x(t), y(t))$, then $y(t) = F_t(x(t))$. 
\end{lemma}
\begin{proof}
If $t \notin K$, then $t \in G_\omega = (\beta_{\omega}, \alpha_{\omega^+})$ for some $\omega$. By the definition of $\gamma$, $y(t) = g_\omega(x(t))$, because $\gamma(t)$ is in the graph of $g_\omega$. Since $F_t = F_{\beta_\omega} = g_\omega = f_{\omega^+} = F_{\alpha_{\omega^+}}$, the result follows.

If $t \in K$ and $\omega_n$ is the defining sequence of $t$, we know by Lemma~\ref{lemma:gammaJ} that $(x(t),y(t)) \in \esslune{L_{\omega_n}}$. By Lemma \ref{lemma:simpleLuneEssential}, $f_{\omega_n}(x(t)) \leq y(t) \leq g_{\omega_n}(x(t))$. Since both bounding sequences converge to $F_t(x(t))$, we're done. 
\end{proof}

\begin{prop}
\label{prop:boundaryGamma}
For each $t \in (0,1]$, $\gamma([0,t]) = \esslune{L}(f,F_t)$.
\end{prop}   
\begin{proof} 
It suffices to show the result for $t \in K$, for if $t \notin K$ and $t_K = \max\{s \in K  \mid s \leq t\}$, then 
$F_t = F_{t_K}$ 
and 
$\gamma([0,t]) = \gamma([0,t] \cap K) = \gamma([0,t_K])$.

Let $t \in K$. We need to show that 
$\gamma([0,t]) = \gamma(K \cap [0,t]) = \esslune{L}(f,F_t).$ 
To see that $\esslune{L}(f,F_t) \subseteq \gamma(K \cap [0,t]),$ 
take $(x,y) \in \interior (L(f,F_t))$, so that $ f(x) < y < F_t(x)$. 
Let $\{\omega_n\}_{n \geq 1}$ be a defining sequence of $(x,y)$ in $L_{(1)} = L$, which is also the defining sequence of some $s \in K$. 
Then $\gamma(s) = (x,y)$. Since $y = \lim f_{\omega_n}(x) = \lim g_{\omega_n}(x) = F_s(x)$, it follows that $F_s(x) < F_t(x)$. 
Lemma~\ref{lemma:ceilingOrder} implies that we necessarily have $F_s \leq F_t$ and thus $s < t$. Consequently, $\interior (L(f,F_t)) \subseteq \gamma(K \cap [0,t])$. Since $\gamma(K \cap [0,t])$ is a compact set (thus closed), $\esslune{L}(f,F_t) = \overline{\interior(L(f,F_t))} \subseteq \gamma(K \cap [0,t]).$

 To see that $\gamma(K \cap [0,t]) \subseteq \esslune{L}(f,F_t)$, write $J_\omega = [\alpha_\omega, \beta_\omega]$ for each $\omega \in \Omega$. Take $s \in K, 0 < s \leq t$, and suppose initially that $s = \alpha_{\hat{\omega}}$ for some $\hat{\omega}$. Since $s > 0$, there exists a unique $\omega$ such that $s = \alpha_{\omega^+}$. Since 
\begin{align*}
\gamma(\alpha_{\omega^+}) &= \lim_{t \nearrow \alpha_{\omega^+}} \gamma(t) \\
&= \lim_{t \nearrow \alpha_{\omega^+}} (\psi_{G_\omega, b_\omega, a_{\omega^+}}(t), g_\omega(\psi_{G_\omega, b_\omega, a_{\omega^+}}(t)))\\
&= \lim_{t \nearrow \alpha_{\omega^+}} (\psi_{G_\omega, b_\omega, a_{\omega^+}}(t), F_{\alpha_{\omega^+}}(\psi_{G_\omega, b_\omega, a_{\omega^+}}(t)))    
\end{align*}
is a limit point of the closed set $\esslune{L}(f,F_s)$, it follows that $\gamma(s) \in \esslune{L}(f,F_s) \subseteq \esslune{L}(f,F_t)$. 

If, on the other hand, $s \notin \{\alpha_\omega\}_{\omega \in \Omega}$ and $\{\omega_n\}$ is the defining sequence for $s$, then $\alpha_{\omega_n} \nearrow s$, thus $\gamma(\alpha_{\omega_n}) \to \gamma(s)$ and $\gamma(\alpha_{\omega_n}) \in \esslune{L}(f,F_{\alpha_{\omega_n}}) \subseteq \esslune{L}(f,F_t)$.   
\end{proof}

We have thus concluded the proof of Proposition~\ref{prop:main}.
To summarize, the continuity of $\gamma$, $F_t$, and $\psi$
are established in Lemma~\ref{lemma:gamma}, Proposition~\ref{prop:continuityCeiling}, and 
Proposition~\ref{prop:continuityOfPsi}, respectively;
property~\ref{item:floorCeiling} comes directly from the definition;
properties~\ref{item:monotonicity} and \ref{item:ceiling} 
are respectively Lemmas~\ref{lemma:ceilingOrder} and \ref{lemma:gammaTinTheCeiling},
property~\ref{item:psi} holds by definition,
property~\ref{item:footprint} is Proposition~\ref{prop:boundaryGamma},
and property~\ref{item:endpoints} is also true by construction.

\section{Proof of the theorem}
\label{sec:maincurve}

The first is step is to obtain a cylindrical version of Proposition~\ref{prop:main}.
Recall that $\T = \R / 2\pi\Z$, and that 
$C^\infty(\T)$ can be regarded as the space of $C^\infty$ $2\pi$-periodic functions $\R \to \R$, endowed with the $C^\infty$ topology (see \S \ref{subsec:initialDefinitions} for details).

\begin{prop}
\label{prop:cylinder}
Given intervals $[t_0,t_1]$ and $[c,d]$, there exist:
\begin{itemize}
	\item a continuous map $\gamma : [t_0,t_1] \to \T \times [c,d]$;
	\item a continuous map $t \in [t_0,t_1] \mapsto F_t \in C^\infty(\T)$;
	\item a continuous function $\psi : \T \times [c,d] \to \R$;
\end{itemize}
with the following properties:
\begin{enumerate}[label=\upshape(\roman*),series=propertyLists]
	\item \label{item:first} $F_{t_0} \equiv c$, $F_{t_1} \equiv d$;
	\item If $t \leq s$ then $F_t \leq F_s$;
	\item Writing $\gamma(t) = (x(t),y(t))$, we have $y(t) = F_t(x(t))$.
	\item \label{item:tangent} $F'_t(x) = \psi(x,F_t(x))$;
	\item \label{item:before_last} for each $t \in (t_0, t_1]$, the image $\gamma([t_0,t])$ equals the closure of the interior of $\{(x,y) \mid x\in \T, \ c \le y \le F_t(x)\}$;
	\item[\upshape(vi')] $\gamma(t_0) = (0,c)$, $\gamma(t_1) = (\pi,d)$; 
\end{enumerate}
\end{prop}

\begin{proof}
It is sufficient to consider the case $[t_0,t_1] = [c,d] = [0,1]$, since the general case follows by rescaling.
Let $g:\R \to [0,1]$ be a $2\pi$-periodic $C^\infty$ function such that $g(0)=0$, $g(\pi)=1$, $g^{(k)}(0) = g^{(k)}(\pi) = 0$ for every $k \ge 1$, and which is strictly monotone in each of the intervals $[0,\pi]$ and $[\pi,2\pi]$.
Consider the following four functions:
\begin{alignat*}{3}
f_1, g_1&:[0,2\pi]        \to \R  &\quad\text{given by}\quad f_1 &\equiv 0,             &\ g_1 &= g|_{[0,2\pi]}, \\
f_2, g_2&:[\pi,3\pi] \to \R  &\quad\text{given by}\quad f_2 &= g|_{[\pi,3\pi]}, &\ g_2 &\equiv 1,
\end{alignat*}
and the corresponding lunes $L_1 = L(f_1,g_1)$ and $L_2 = L(f_2,g_2)$. 
Applying Proposition~\ref{prop:main} to the lune $L_i$ we obtain
a Peano curve $\gamma_i$, a family of ceiling functions $F_{i,t}$, and a function $\psi_i$.
Let $p: L_1 \cup L_2 \to \T \times [0,1]$ be the bijective map defined by $p(x,y) = (x \bmod 2\pi, y)$; set
$$
\gamma^*(t) = \begin{cases} \gamma_1(3t), &0 \le t < 1/3; \\
\left( 3\pi(1-t), g\left(3\pi(1-t)\right)\right), &1/3 \le t \le 2/3; \\
\gamma_2(3t-2), &2/3 < t \le 1;
\end{cases} 
$$
and $\gamma = p \circ \gamma^*$;
set $\psi = \psi^* \circ p^{-1}$ for $\psi^*: L_1 \cup L_2 \to \R$ such that $\psi^*|{L_1} = \psi_1$ and $\psi^*|{L_2} = \psi_2$; finally,
for $x \in \T$, let
$$
F_t(x) = \begin{cases} F_{1,3t}(x), &0 \le t < 1/3; \\
g(x), &1/3 \le t \le 2/3; \\
F_{2,3t-2}(x), &2/3 < t \le 1.
 \end{cases} 
$$
Then $\gamma$, $\psi$ and $F_t$ have the desired properties.
\end{proof}

We improve the previous proposition by controlling the derivatives:

\begin{prop}
\label{prop:cylinder_with_norm}
Given intervals $[t_0,t_1]$, $[c,d]$
and numbers $k_0 \in \N$, $\delta_0 > 0$, there exist
maps $\gamma$, $F_t$, and $\psi$
satisfying properties \ref{item:first}--\ref{item:before_last} in Proposition~\ref{prop:cylinder}
and, in addition, the following ones:
\begin{enumerate}[resume=propertyLists,label=\upshape(\roman*)]
	\item \label{item:endsMeet} $\gamma(t_0) = (0,c)$, $\gamma(t_1) = (0,d)$; 
	\item $\| F_t - c_t \|_{k_0} < \delta_0$ for every $t$, where $c_t$ is the constant $\frac{c(t_1-t) + d(t - t_0)}{t_1 - t_0}$. \label{item:controlledNorm} 
\end{enumerate}
\end{prop}

\begin{proof}
Again, it is sufficient to consider the case $[t_0,t_1] = [c,d] = [0,1]$.
Let $\hat{\gamma}$, $\hat{F}_t$, and $\hat{\psi}$ be given by the previous proposition.
By compactness, we have $\|\hat{F}_t\|_{k_0} \leq C$ for some finite $C$ independent of $t$. Fix an odd integer $n > (C+1)/\delta_0$. By rescaling and translating we obtain $\hat\gamma_j$, $\hat F_{j,t}$, and  $\hat\psi_j$ for the cylinders $\T \times [j/n, (j+1)/n]$, where $j=0,1,\ldots n-1$, with the extra property that if $t \in [j/n, (j+1)/n]$ then $\|\hat F_{j,t} - j/n\|_{k_0} \leq C/n < \delta_0 - 1/n$.
Finally, we rotate the cylinders so that everything glues: in other words, 
for each $t \in [0,1]$ we let $j = \floor{nt}$ and define:
$$
\gamma(t)  = \hat\gamma_j(nt) + (j\pi,0) \, , \quad
F_t(x)     = \hat F_{j,nt-j}(x + j\pi)  \, ,  \quad 
\psi(x,t)  = \hat \psi_j(x + j\pi,t) \, .
$$
It is clear that these maps have the required properties \ref{item:first}--\ref{item:before_last} and \ref{item:endsMeet}.
To see property~\ref{item:controlledNorm}, notice that $c_t = t$ so letting $j = \floor{nt}$ we have
\begin{equation*}
\|F_t - c_t\|_{k_0} \le \|F_t - j/n\|_{k_0} + 1/n < \delta_0 \, . \qedhere
\end{equation*}
\end{proof}

Finally, we explain how the previous proposition allows us to conclude:

\begin{proof}[Proof of the theorem]
Let $k :(0,\infty) \to \N$ be upper semicontinuous
and $\epsilon:(0,\infty) \to (0,\infty)$ be lower semicontinuous.
Then there exist
two-sided sequences $\{t_n\}_{n\in\Z}$, $\{k_n\}_{n\in \Z}$, and $\{\epsilon_n\}_{n\in \Z}$ 
taking values in $(0,\infty)$, $\N$, and $(0,\infty)$ respectively, 
such that $\{t_n\}$ is monotonically increasing, 
$\lim_{n \to -\infty} t_n = 0$, $\lim_{n \to +\infty} t_n = \infty$, 
and 
$$
t\in [t_n, t_{n+1}] \ \Rightarrow \epsilon(t) \ge \epsilon_n \text{ and } k(t) \le k_n.
$$
For each $n$, let $\delta_n = \epsilon_n/2^{k_n}$,
and apply Proposition~\ref{prop:cylinder_with_norm} with both intervals equal to 
$I_n = [t_n, t_{n+1}]$,
thus obtaining a Peano curve $\gamma_n: I_n \to \T \times I_n$, a family of ceiling functions $F_{n,t}$ and a continuous function $\psi_n$ defined on $\T \times I_n$ satisfying all seven properties. Also, notice that $c_t = t$ in part~\ref{item:controlledNorm}.

Define a diffeomorphism
$P \colon \T \times (0, \infty) \to \R^2 \setminus \{(0,0)\}$
by $P(\theta, r) = (r \cos \theta, r \sin \theta)$.
Recall that $\alpha_t(\theta) = P(\theta,t)$.
We now construct maps
$\gamma^*: (0,\infty) \to \T \times (0,\infty)$, $F_t^*: \T \to (0,\infty)$ and $\psi^*:\T \times (0,\infty) \to \R$ by setting for $(x,t) \in \T \times [t_n, t_{n+1}]$:
$$
\gamma^*(t) = \gamma_n(t), \quad
F_t^*(x)    = F_{n,t}(x), \quad
\psi(x,t)   = \psi_n(x,t) \, .
$$

Next, we define the Peano curve $\gamma$ by $\gamma(0) = 0$ and $\gamma(t) = P \circ \gamma^*(t)$ for $t > 0$. 
Let $\Lambda^*$ be the line field on $\T \times (0,\infty)$ spanned by the vector field $\frac{\partial}{\partial \theta} + \psi^*(r,\theta) \frac{\partial}{\partial r}$;
by pushing it forward by the derivative of $P$, we obtain a  line field $\Lambda$ on $\R^2 \setminus \{(0,0)\}$.

Notice that, for each $t > 0$, $\gamma^*([0,t]) =\{(x,y): x \in \T, 0 < y \leq F_t(x)\}$. It follows then that $\beta_t: \theta \in \T \mapsto P(\theta,F_t(\theta)) \in \R^2$ is a smooth embedding whose image is $\partial\gamma([0,t])$. By property~\ref{item:tangent} of Proposition~\ref{prop:cylinder}, $\beta_t$ is tangent to the line field $\Lambda$.

To conclude the proof we check that the proximity condition \eqref{eq:proximity} is satisfied.
Since $\beta_t(\theta) = F_t(\theta)\alpha_1(\theta)$ and $\alpha_t(\theta) = t \alpha_1(\theta)$, we have, for each $k \in \N$,
$$
\|\beta_t^{(k)}(\theta) - \alpha_t^{(k)}(\theta)\| 
\le \sum_{i=0}^k \binom{k}{i} \big|(F_t -t)^{(i)}(\theta)\big| \,  \underbrace{\big|\alpha_1^{(k-i)}(\theta)\big|}_{1}
\le 2^k \|F_t - t\|_k \, . 
$$
Given $t>0$, let $n$ be such that $t \in I_n$. Then
\begin{equation*}
\|\alpha_t - \beta_t\|_{k(t)} 
\le \|\alpha_t - \beta_t\|_{k_n} 
\le  2^{k_n} \|F_t - t \|_{k_n}\\
< 2^{k_n} \delta_n 
= \epsilon_n 
\le \epsilon(t).   \qedhere
\end{equation*}
\end{proof}

\bibliography{biblio}
\bibliographystyle{acm}
\end{document}